\documentclass[12pt,twoside]{amsart}
\usepackage{amssymb}
\usepackage{amscd}
\usepackage{setspace}

\title[Effective freeness II]{Effective base point free theorem for log canonical pairs II\\---Angehrn--Siu type 
theorems---}
\author{Osamu Fujino} 
\date{2009/7/21}
\address{Department of Mathematics, Faculty of Science, 
Kyoto University, 
Kyoto 606-8502 Japan} 
\email{fujino@math.kyoto-u.ac.jp}
\newcommand{\Supp}[0]{{\operatorname{Supp}}}

\newcommand{\mult}[0]{{\operatorname{mult}}}

\newtheorem{thm}{Theorem}[section]
\newtheorem{lem}[thm]{Lemma}

\newtheorem{prop}[thm]{Proposition}

\theoremstyle{definition}

\newtheorem{defn}[thm]{Definition}
\newtheorem{rem}[thm]{Remark}
\newtheorem*{ack}{Acknowledgments}      
\newtheorem*{notation}{Notation}         
\newtheorem{say}[thm]{}
\newtheorem{case}{Case}
\begin{document}
\bibliographystyle{amsalpha+}

\maketitle



\section{Introduction}

The main purpose of this paper is to 
advertise the power of 
the new cohomological technique introduced in \cite{ambro}. 
By this new method, we generalize Angehrn--Siu type 
effective base point freeness and 
point separation (see \cite{as} and \cite[5.8 and 5.9]{kollar2}) 
for {\em{log canonical}} pairs. 
Here, we adopt Koll\'ar's formulation in \cite{kollar2} because 
it is suitable for singular varieties. 
The main ingredients 
of our proof are the inversion 
of adjunction on log canonicity (see 
\cite{kawakita}) and 
the new cohomological technique 
(see \cite{ambro}). 
For the Koll\'ar type effective freeness for 
log canonical pairs, see \cite{f-new}. 
In \cite{fuji-non}, we give a simple new proof of the base point 
free theorem for log canonical pairs. 
It is closely related to the arguments in this paper. 

The following theorems are the main theorems of this paper. 

\begin{thm}[{Effective 
Freeness, cf.~\cite[5.8 Theorem]{kollar2}}]\label{main-as} 
Let $(X, \Delta)$ be a projective {\em{log canonical}} pair and $M$ a 
line bundle on $X$. 
Assume that $M\equiv K_X+\Delta+N$, where $N$ is an ample 
$\mathbb Q$-divisor 
on $X$. 
Let $x\in X$ be a closed point and assume that there are 
positive numbers $c(k)$ with 
the following properties{\em{:}} 
\begin{enumerate}
\item If $x\in Z\subset X$ is an irreducible 
$($positive dimensional$)$ 
subvariety, then 
$$
(N^{\dim Z}\cdot Z)>c(\dim Z)^{\dim Z}. 
$$ 
\item The numbers $c(k)$ satisfy the inequality 
$$
\sum _{k=1}^{\dim X} \frac{k}{c(k)}\leq 1. 
$$
\end{enumerate} 
Then $M$ has a global section not vanishing at $x$. 
\end{thm}

\begin{thm}[{Effective Point 
Separation, cf.~\cite[5.9 Theorem]{kollar2}}]\label{main-as2} 
Let $(X, \Delta)$ be a projective 
{\em{log canonical}} pair and $M$ a 
line bundle on $X$. 
Assume that $M\equiv K_X+\Delta+N$, where $N$ is an ample 
$\mathbb Q$-divisor 
on $X$. 
Let $x_1, x_2\in X$ be closed points and assume that there are 
positive numbers $c(k)$ with 
the following properties{\em{:}} 
\begin{enumerate}
\item If $Z\subset X$ is an 
irreducible $($positive dimensional$)$ 
subvariety which contains $x_1$ or $x_2$, then 
$$
(N^{\dim Z}\cdot Z)>c(\dim Z)^{\dim Z}. 
$$ 
\item The numbers $c(k)$ satisfy the inequality 
$$
\sum _{k=1}^{\dim X}
\sqrt[\leftroot{0}\uproot{0} k]{2}\frac{k}{c(k)}
\leq 1. 
$$
\end{enumerate} 
Then global sections of $M$ separate $x_1$ and $x_2$. 
\end{thm}

To strengthen the above theorems, we introduce the 
following definition. 

\begin{defn}
Let $X$ be a normal variety and $B$ an effective $\mathbb Q$-divisor 
on $X$ such that $K_X+B$ is $\mathbb Q$-Cartier. 
Let $x\in X$ be a closed point. If $(X, B)$ is kawamata log 
terminal at $x$, then 
we put $\mu(x, X, B)=\dim X$. 
When $(X, B)$ is log canonical but not kawamata log terminal at $x$, 
we define 
\begin{align*}
&\mu(x, X, B)\\ & =\min \{ \dim W\,  |\,  
W \ {\text{is an lc center of}}\  (X, B) \ {\text{such that}} \ 
x\in W \}.
\end{align*} 
If $(X, B)$ is not log canonical at $x$, then 
we do not define $\mu(x, X, B)$ for such $x$. 
For the details of lc centers, see Theorem \ref{a1} below. 
\end{defn}

Then we obtain the following slight generalizations 
of the above theorems. 

\begin{rem}\label{re1}
In Theorem \ref{main-as}, 
if $\mu =\mu (x, X, \Delta)<\dim X$ 
and $W$ is the minimal lc center of $(X, B)$ with 
$x\in W$, then 
we can weaken the assumptions as follows. 
If $Z\subset W$ is an irreducible 
$($positive dimensional$)$ 
subvariety which contains $x$, then 
$$
(N^{\dim Z}\cdot Z)>c(\dim Z)^{\dim Z} 
$$ 
and the numbers $c(k)$ satisfy the inequality 
$$
\sum _{k=1}^{\mu} \frac{k}{c(k)}\leq 1. 
$$
In particular, if $\mu(x, X, \Delta)=0$, 
then we need no assumptions 
on $c(k)$. 
\end{rem}

\begin{rem}\label{re2}
In Theorem \ref{main-as2}, 
we put $\mu_1=\mu (x_1, X, \Delta)$ and 
$\mu_2=\mu (x_2, X, \Delta)$. 
Possibly after switching $x_1$ and $x_2$, we can 
assume that $\mu_1\leq \mu_2$. 
Let $W_1$ (resp.~$W_2$) be the minimal 
lc center of $(X, \Delta)$ such that 
$x_1\in W_1$ (resp.~$x_2\in W_2$) when 
$\mu_1<\dim X$ (resp.~$\mu_2<\dim X$). 
Otherwise, we put $W_1=X$ (resp.~$W_2=X$).  
\begin{enumerate}
\item[(i)] If $W_1\not \subset W_2$, then 
the assumptions 
in Theorem \ref{main-as2} 
can be replaced as follows. 
If $x_1\in Z\subset W_1$ is an irreducible 
$($positive dimensional$)$ 
subvariety, then 
$$
(N^{\dim Z}\cdot Z)>c(\dim Z)^{\dim Z} 
$$ 
and the numbers $c(k)$ satisfy the inequality 
$$
\sum _{k=1}^{\mu_1} \frac{k}{c(k)}\leq 1. 
$$
\item[(ii)] If $W_1\subsetneq W_2$, then 
the assumptions 
in Theorem \ref{main-as2} 
can be replaced as follows. 
If $x_2\in Z\subset W_2$ is an irreducible 
$($positive dimensional$)$ 
subvariety, then 
$$
(N^{\dim Z}\cdot Z)>c(\dim Z)^{\dim Z} 
$$ 
and the numbers $c(k)$ satisfy the inequality 
$$
\sum _{k=1}^{\mu_2} \frac{k}{c(k)}\leq 1. 
$$
\item[(iii)] If $W_1=W_2$, then 
we can weaken the assumptions in 
Theorem \ref{main-as2} as follows. 
If $Z\subset W_1=W_2$ is an irreducible 
$($positive dimensional$)$ 
subvariety which contains $x_1$ or $x_2$, then 
$$
(N^{\dim Z}\cdot Z)>c(\dim Z)^{\dim Z} 
$$ 
and the numbers $c(k)$ satisfy the inequality 
$$
\sum _{k=1}^{\mu}
\sqrt[\leftroot{0}\uproot{0} k]{2}\frac{k}{c(k)}
\leq 1, 
$$ 
where $\mu=\mu_1=\mu_2$. 
\end{enumerate}
\end{rem}

\begin{say}
Let us quickly review 
the usual technique for base point free theorems via 
multiplier ideal sheaves (cf.~\cite{as} and \cite{kollar2}). 
Let $(X, \Delta)$ be a projective log canonical 
pair and $M$ a line bundle 
on $X$. 
Assume that $M\equiv K_X+\Delta+N$, where 
$N$ is an ample $\mathbb Q$-divisor on $X$. 
Let $x\in X$ be a closed point. 
Assume that $(X, \Delta)$ is kawamata log terminal 
around $x$. 
In this case, it is sufficient to find an effective 
$\mathbb Q$-divisor $E$ on $X$ such 
that $E\equiv cN$ for $0<c<1$, 
$(X, \Delta +E)$ is log canonical around $x$, 
and that $x$ is an isolated non-klt locus of $(X, \Delta +E)$. 
Once we obtain $E$, we have $H^1(X, M\otimes 
\mathcal J(X, \Delta +E))=0$ by the Kawamata--Viehweg--Nadel 
vanishing theorem by $M-(K_X+\Delta+E)\equiv (1-c)N$, 
where $\mathcal J(X, \Delta +E)$ is the multiplier ideal sheaf associated 
to $(X, \Delta +E)$. Therefore, 
the restriction map 
$H^0(X, M)\to H^0(X, M\otimes \mathcal O_X/\mathcal J(X, \Delta 
+E))$ is surjective. 
By the choice of $E$, $x$ is isolated in $\Supp (\mathcal O_X/ 
\mathcal J(X, \Delta+E))$. So, we can obtain a section of $M$ which 
does not vanish at $x$. 
To construct $E$, we need the 
Ohsawa--Takegoshi $L^2$-extension theorem 
or the inversion of adjunction. For the details, see \cite{as} and \cite{kollar2}. 

From now on, we assume that $(X, \Delta)$ is log canonical 
but not kawamata log terminal at $x$. 
When $x$ is an isolated non-klt locus of $(X, \Delta)$, 
we can apply the above arguments. 
However, in general, $x$ is not isolated 
in $\Supp (\mathcal O_X/\mathcal J(X, \Delta))$. 
So, we can not directly use the techniques in \cite{as} and \cite{kollar2}. 
Fortunately, by using 
the new framework introduced in \cite{ambro}, 
we know that it is sufficient 
to find an effective $\mathbb Q$-divisor 
$E$ on $X$ such that $E\equiv cN$ for $0<c<1$, 
$(X, \Delta+E)$ is log canonical 
around $x$, and that $x$ is an lc center of $(X, \Delta+E)$. 
It is because we can prove that the restriction map $H^0(X, M)\to 
\mathbb C(x)\otimes M$ is surjective once we obtain such 
$E$ (see Theorem \ref{a2}). 
We note that the inversion of adjunction on log canonicity 
plays a crucial role when we construct $E$. 
\end{say}

We summarize the contents of this paper. 
In Section \ref{sec2}, 
we will explain the proof of 
Theorems \ref{main-as} and \ref{main-as2}. 
It is essentially the same as Koll\'ar's in \cite[Section 6]{kollar2} 
if we adopt the new cohomological technique and 
the inversion of adjunction on log canonicity. 
So, we will omit some details in Section \ref{sec2}. 
In Section \ref{sec-app}, which is an appendix, 
we collect some basic properties of 
lc centers and the new cohomological technique for 
the reader's convenience since they are not popular yet. 

I hope that this paper 
and \cite{f-new} will motivate the reader to study 
the new cohomolocial technique. 
For a systematic and thorough treatment on 
this topic, that is, the new cohomolocial technique and 
the theory of 
quasi-log varieties, we recommend the reader to 
see \cite{book}. 

\begin{ack}
I was partially supported by the Grant-in-Aid for Young Scientists 
(A) $\sharp$20684001 from JSPS. I was 
also supported by the Inamori Foundation. 
I thank the referee for useful comments. 
\end{ack}

\begin{notation}
We will work over the complex number field $\mathbb C$ throughout  
this paper. 
{\em{Numerical equivalence}} of line bundles and 
$\mathbb Q$-Cartier $\mathbb Q$-divisors 
is denoted by $\equiv$. 
{\em{Linear equivalence}} of Cartier divisors 
is denoted by $\sim$. 
Let $X$ be a normal variety and $B$ an effective 
$\mathbb Q$-divisor such that 
$K_X+B$ is $\mathbb Q$-Cartier. 
Then we can define the discrepancy 
$a(E, X, B)\in \mathbb Q$ for 
every prime divisor $E$ over $X$. 
If $a(E, X, B)\geq -1$ (resp.~$>-1$) for 
every $E$, then $(X, B)$ is called {\em{log canonical}} 
(resp.~{\em{kawamata log terminal}}). 
Note that there always exists the maximal 
Zariski open set $U$ of 
$X$ such that $(X, B)$ is log canonical on $U$. 
If $E$ is a prime divisor over $X$ such that 
$a(E, X, B)=-1$ and the closure of the 
image 
of $E$ on $X$, which is denoted by $c_X(E)$ and called 
the {\em{center}} of $E$ on $X$, is not contained in $X\setminus U$, 
then $c_X(E)$ is called a {\em{center of log canonical 
singularities}} or {\em{log canonical center}} ({\em{lc center}}, for short) 
of $(X, B)$. 
Let $x\in X$ be a closed point. 
Assume that $(X, B)$ is log canonical at $x$ but not 
kawamata log terminal. 
Then there is a unique minimal log canonical 
center $W$ passing through $x$, and $W$ is normal at $x$. 
(see Theorem \ref{a1} in the 
appendix). 
\end{notation}

\section{Proof of the main theorem}\label{sec2}

The main results of this section are 
the following propositions. 

\begin{prop}[{cf.~\cite[6.4 Theorem]{kollar2}}]\label{12}
Let $(X, \Delta)$ be a projective log canonical 
pair and $N$ an ample $\mathbb Q$-divisor 
on $X$. 
Let $x\in X$ be a closed point and $c(k)$ positive numbers such that 
if $x\in Z\subset X$ is an irreducible $($positive dimensional$)$ 
subvariety then 
$$
(N^{\dim Z}\cdot Z)>c(\dim Z)^{\dim Z}. 
$$ 
Assume that 
$$
\sum _{k=1}^{\dim X} \frac{k}{c(k)}\leq 1. 
$$
Then there is an effective $\mathbb Q$-divisor 
$D\equiv c N$ with $0<c<1$ and an open 
neighborhood $x\in X^0\subset X$ such that 
\begin{enumerate}
\item $(X^0, \Delta +D)$ is log canonical, and 
\item $x$ is a center of log canonical singularities for the 
pair $(X, \Delta+D)$. 
\end{enumerate}
\end{prop}

\begin{rem}
In Proposition \ref{12}, the assumptions on $Z$ and 
$c(k)$ can be replaced as in Remark \ref{re1}. 
\end{rem}

\begin{prop}[{cf.~\cite[6.5 Theorem]{kollar2}}]\label{122}
Let $(X, \Delta)$ be a projective log canonical 
pair and $N$ an ample $\mathbb Q$-divisor 
on $X$. 
Let $x_1, x_2\in X$ be closed points and $c(k)$ positive numbers such that 
if $Z\subset X$ is an irreducible $($positive dimensional$)$ 
subvariety such that $x_1\in Z$ or $x_2\in Z$ then 
$$
(N^{\dim Z}\cdot Z)>c(\dim Z)^{\dim Z}. 
$$ 
Assume also that 
$$
\sum _{k=1}^{\dim X}
\sqrt[\leftroot{0}\uproot{0} k]{2}\frac{k}{c(k)}
\leq 1. 
$$
Then, possibly after switching 
$x_1$ and $x_2$, one can take an effective $\mathbb Q$-divisor 
$D\equiv c N$ with $0<c<1$ and an open 
neighborhood $x_1\in X^0\subset X$ such that 
\begin{enumerate}
\item $(X^0, \Delta +D)$ is log canonical, 
\item $x_1$ is a center of log canonical singularities for the 
pair $(X, \Delta+D)$, and 
\item $(X, \Delta +D)$ is not log canonical at $x_2$. 
\end{enumerate}
\end{prop}

\begin{rem}
In Proposition \ref{122}, the assumptions on $Z$ and 
$c(k)$ can be replaced as in Remark \ref{re2}. 
\end{rem}
First, we give a proof of Theorem \ref{main-as} 
by using Proposition \ref{12}. 

\begin{proof}[Proof of {\em{Theorem \ref{main-as}}}]
We consider the pair $(X, \Delta +D)$ constructed 
in Proposition \ref{12}. It is not necessarily log canonical 
but has a natural quasi-log structure (see Theorem \ref{a2}). 
Let $X\setminus X_{-\infty}$ be the maximal 
Zariski open set where $(X, \Delta +D)$ 
is log canonical. 
Since $(X, \Delta +D)$ is log canonical around $x$ and $x$ is 
an lc center of $(X, \Delta +D)$, 
the union of $x$ and $X_{-\infty}$ 
has a natural quasi-log structure 
$X'$ induced by the quasi-log structure 
of $(X, \Delta +D)$ (see Theorem \ref{a2}). 
We consider the following 
short exact sequence 
$$
0\to \mathcal I_{X'}\to \mathcal O_X\to \mathcal O_{X'}\to 0. 
$$ 
Since $M-(K_X+\Delta +D)\equiv (1-c)N$ is ample, 
$H^1(X, \mathcal I_{X'}\otimes M)=0$ by the 
vanishing theorem (see Theorem \ref{a2}). 
Therefore, $H^0(X, M)\to H^0(X', M)$ is surjective. 
We note that $x$ is isolated in $X'$ because $x\not\in 
X_{-\infty}$. 
Thus, the evaluation map $H^0(X, M)\to 
M\otimes \mathbb C(x)$ is surjective. 
This is what we wanted. 
\end{proof}

Next, we give a proof of Theorem \ref{main-as2} by 
Proposition \ref{122}. 

\begin{proof}[Proof of {\em{Theorem \ref{main-as2}}}] 
We use the same notation as in the 
proof of Theorem \ref{main-as}. 
In this case, $x_2$ is on $X_{-\infty}$. 
In particular, $x_2$ is a point of $X'$. 
Since $H^0(X, M)\to H^0(X', M)$ is surjective and 
$x_1$ is an isolated point of $X'$, 
we can take a section of $M$ which separates $x_1$ and $x_2$. 
\end{proof}

Therefore, all we have to prove are Propositions \ref{12} and \ref{122}. 
Let us recall the following easy but important result. 
Here, we need the inversion of adjunction on log canonicity. 
 
\begin{prop}[{cf.~\cite[6.7.1.~Theorem]{kollar2}}]\label{13}
Let $(X, \Delta)$ be a projective log canonical 
pair with $\dim X=n$ and $x\in X$ a closed point. 
Let $H$ be an ample $\mathbb Q$-divisor 
on $X$ such that 
$(H^n)>n^n$. Then 
there is an effective $\mathbb Q$-divisor 
$B_x\equiv H$ such that 
$(X, \Delta +B_x)$ is not log canonical 
at $x$. 
\end{prop}
\begin{proof}
If we adopt Lemma \ref{inver} below, then the proof of 
\cite[6.7.1.~Theorem]{kollar2} works without any changes. 
\end{proof}

\begin{lem}[{cf.~\cite[7.8 Corollary]{kollar2}}]\label{inver}
Let $(X, \Delta)$ be a log canonical pair 
and $B_c${\em{:}}~$c\in C$ an algebraic family 
of $\mathbb Q$-divisors on $X$ parametrized by 
a smooth 
pointed curve $0\in C$. Assume that 
$(X, \Delta +B_0)$ is log canonical 
at $x\in X$. Then 
there is a Euclidean open neighborhood 
$x\in W\subset X$ such that 
$(X, \Delta +B_c)$ is log canonical 
on $W$ for $c\in C$ near $0$. 
\end{lem}
Lemma \ref{inver} is a direct consequence of 
Kawakita's inversion of adjunction on log canonicity (see 
\cite[Theorem]{kawakita}). 
The next proposition is a reformulation of 
\cite[6.8.1.~Theorem]{kollar2}. In Koll\'ar's proof 
in \cite[Section 6]{kollar2}, 
he cuts down the non-klt locus. On the other hand, 
we cut down the minimal lc center passing through $x$. 
The advantage of our method is in the fact 
that the minimal lc center is always {\em{irreducible}} 
by its definition. 
So, we do not need to use {\em{tie breaking technique}} (see 
\cite[6.9 Step 3]{kollar2}) to make the non-klt locus 
irreducible even when $(X, \Delta)$ is kawamata 
log terminal. 

\begin{prop}[{cf.~\cite[6.8.1.~Theorem]{kollar2}}]\label{14} 
Let $(X, \Delta)$ be a projective log canonical 
pair and $x\in X$ a closed 
point. Let $D$ be an effective $\mathbb Q$-Cartier 
$\mathbb Q$-divisor 
on $X$ such that $(X, \Delta+D)$ is log canonical in 
a neighborhood of $x$. 
Assume that $Z$ is the minimal lc center of $(X, \Delta+D)$ such that 
$x\in Z$ with $k=\dim Z>0$. 
Let $H$ be an ample $\mathbb Q$-divisor 
on $X$ such that $(H^k\cdot Z)>k^k$. 
Then there are an effective $\mathbb Q$-divisor 
$B\equiv H$ and a rational 
number $0<c<1$ such that 
\begin{enumerate}
\item $(X, \Delta+D+cB)$ is log canonical 
in a neighborhood of $x$, and 
\item there is the minimal lc center $Z_1$ of $(X, \Delta +D +cB)$ 
such that $x\in Z_1$ and $\dim Z_1<\dim Z$. 
\end{enumerate}
\end{prop}

\begin{proof}
By the assumption, there are a projective birational 
morphism $f:Y\to X$ and a divisor $E\subset Y$ such that 
$a(E, X, \Delta+D)=-1$ and $f(E)=Z$. 
We write $K_Y=f^*(K_X+\Delta+D)+\sum e_i E_i$ where 
$E=E_1$ and so $e_1=-1$. 
Let $Z^0\subset Z$ be an open subset such that 
\begin{enumerate}
\item $Z^0$ is smooth and $f|_E:E\to Z$ is smooth over $Z^0$, and 
\item if $z\in Z^0$, then 
$(f|_E)^{-1}(z)\not\subset E_i$ for $i\ne 1$. 
\end{enumerate}

\begin{lem}[{cf.~\cite[6.8.3 Claim]{kollar2}}]\label{24}
Notation as above. 
Choose $m\gg 1$ such that $mH$ is Cartier. 
Let $U$ be the Zariski open subset of $X$ where $(X, \Delta+D)$ is log canonical. 
Then, for every $z\in Z^0$, the following 
assertions hold{\em{:}} 
\begin{enumerate}
\item There is a divisor $F_z\sim mH|_{Z}$ such that 
$\mult _z F_z>mk$. 
\item $\mathcal O_X(mH)\otimes I_Z$ is generated by global sections and 
$$H^1(X, \mathcal O_X(mH)\otimes I_Z)=0.$$  
In particular, $H^0(X, \mathcal O_X(mH))\to H^0(Z, \mathcal O_Z(mH|_Z))$ 
is surjective. 
\item For any $F\sim mH|_Z$, there is $F^X\sim mH$ such that 
$F^X|_Z=F$ and $(X, \Delta +D+(1/m)F^X)$ is log canonical 
on $U\setminus Z$. 
\item Let $F^X_z\sim mH$ be such that ${F^X_z}|_Z=F_z$. 
Then $(X, \Delta +D+(1/m)F^X_z)$ is not log canonical at $z$.  
\end{enumerate}
\end{lem}

The proof of Lemma \ref{24} is the same as the 
proof of \cite[6.8.3 Claim]{kollar2}. 
So, we omit it here. Pick $z_0\in Z$ arbitrary. 
Let $0\in C$ be a smooth affine curve and $g:C\to Z$ 
a morphism 
such that $z_0=g(0)$ and $g(c)\in Z^0$ for 
general $c\in C$. 
For general $c\in C$, pick $F_c:=F_{g(c)}$ as in Lemma \ref{24} (1). 
Let $F_0=\lim _{c\to 0}F_c$. Then 
we obtain the following lemma. 
For the precise meaning of $\lim_{c\to 0}F_c$, 
see the proof of \cite[6.7.1.~Theorem]{kollar2}. 

\begin{lem}[{cf.~\cite[6.8.4 Claim]{kollar2}}]\label{25} 
Notation as above. 
There exists a divisor 
$F^X_0\in |mH|$ such that 
\begin{enumerate}
\item ${F^X_0}|_Z=F_0$, 
\item $(X, \Delta+D+(1/m)F^X_0)$ is log canonical 
on $U\setminus Z$, 
\item $(X, \Delta+D+(1/m)F^X_0)$ is log canonical at 
the generic point of $Z$, and 
\item $(X, \Delta +D+(1/m)F^X_0)$ is not log canonical 
at $z_0$. 
\end{enumerate}
\end{lem}
The proof of Lemma \ref{25} is the same 
as the proof of \cite[6.8.4 Claim]{kollar2} 
if we adopt Lemma \ref{inver}. 
To finish the proof of Proposition \ref{14}, we 
set $B=(1/m)F^X_0$. 
Let $c$ be the maximal value such that 
$(X, \Delta +D+cB)$ is log canonical 
at $x$. Then we have the desired properties. 
\end{proof}

\begin{proof}[Proof of {\em{Proposition \ref{12}}}] 
Without loss of generality, we can assume that 
$c(k)\in \mathbb Q$ for every $k$. 
If $(X, \Delta)$ is kawamata log terminal around $x$, 
then we put $Z_1=X$. 
Otherwise, 
let $Z_1$ be the minimal lc center of $(X, \Delta)$ such that 
$x\in Z_1$. If $\dim Z_1=k_1>0$, then 
we can find $x\in D_1\equiv \frac{k_1}{c(k_1)}N$ and 
$0<c_1<1$ such that $(X, \Delta+c_1D_1)$ 
is log canonical around $x$ and that $k_2=\dim Z_2<k_1$, 
where $Z_2$ is the minimal lc center of $(X, \Delta+c_1D_1)$ with 
$x\in Z_2$ (see Proposition \ref{14}). Repeat this process. 
Then we can find 
$n\geq \mu_1=k_1>k_2>\cdots >k_l>0$, where $k_i\in \mathbb Z$, with the 
following properties. 
\begin{enumerate}
\item there exists an effective $\mathbb Q$-divisor $D_i$ such that 
$D_i\equiv\frac{k_i}{c(k_i)}N$ for every $i$, 
\item there exists a rational number $c_i$ with $0<c_i<1$ for every $i$, 
\item $(X, \Delta +\sum _{i=1}^{l}c_iD_i)$ is log canonical 
around $x$, and  
\item $x$ is an lc center of the pair $(X, \Delta +\sum _{i=1}^{l}c_iD_i)$. 
\end{enumerate}  
We put $D=\sum _{i=1}^{l}c_i D_i$. 
Then $D$ has the desired properties. 
We note that 
$0<c=\sum _{i=1}^{l}c_i\frac{k_i}{c(k_i)}<1$ and $D\equiv cN$. 
\end{proof}

From now on, we consider the proof of Proposition \ref{122}. 
We can prove Proposition \ref{27} 
by the same way as Proposition \ref{13}. 

\begin{prop}\label{27} 
Let $(X, \Delta)$ be a projective log canonical pair with $\dim X=n$ and 
$x, x'\in X$ closed points. 
Let $H$ be an ample $\mathbb Q$-divisor on $X$ such that 
$(H^n)>2n^n$. Then 
there is an effective 
$\mathbb Q$-divisor $B_{x, x'}\equiv H$ 
such that 
$(X, \Delta +B_{x, x'})$ is not log canonical 
at $x$ and $x'$.  
\end{prop}

By Proposition \ref{27}, we can modify Proposition \ref{14} as follows. 
We leave the details as an exercise for the 
reader. 

\begin{prop}\label{144} 
Let $(X, \Delta)$ be a projective log canonical 
pair and $x, x'\in X$ closed 
points. Let $D$ be an effective $\mathbb Q$-Cartier $\mathbb Q$-divisor 
on $X$ such that $(X, \Delta+D)$ is log canonical in 
a neighborhood of $x$ and $x'$. 
Assume that there exists a 
minimal lc center $Z$ of $(X, \Delta+D)$ such that 
$x, x'\in Z$ with $k=\dim Z>0$. 
Let $H$ be an ample $\mathbb Q$-divisor 
on $X$ such that $(H^k\cdot Z)>2k^k$. 
Then there are an effective $\mathbb Q$-divisor 
$B\equiv H$ and a rational 
number $0<c<1$ such that 
\begin{enumerate}
\item $(X, \Delta+D+cB)$ is not kawamata log terminal 
at the points $x$ and $x'$, and 
is log canonical at one of them, say at $x$, and 
\item there is the minimal lc center $Z_1$ of $(X, \Delta +D +cB)$ 
such that $x\in Z_1$ and $\dim Z_1<\dim Z$. 
\end{enumerate}
\end{prop}

The next proposition is very easy. 

\begin{prop}\label{145} 
Let $(X, \Delta)$ be a projective log canonical 
pair and $x, x'\in X$ closed 
points. Let $D$ be an effective $\mathbb Q$-Cartier $\mathbb Q$-divisor 
on $X$ such that $(X, \Delta+D)$ is log canonical in 
a neighborhood of $x$ and $x'$. 
Assume that the minimal lc center $Z\ni x'$ of $(X, \Delta+D)$ does not 
contain $x$, that is, $x\not \in Z$. 
Let $H$ be an ample $\mathbb Q$-divisor 
on $X$. 
Then there is an effective $\mathbb Q$-divisor 
$B\equiv H$ 
such that $x\not \in B$ and $(X, \Delta+\varepsilon B)$ is not 
log canonical at $x'$ for every $\varepsilon >0$. 
\end{prop}

\begin{proof}
We take a general member $A$ of 
$H^0(X, \mathcal O_X(lH)\otimes I_{Z})$, 
where $l$ is sufficiently large and divisible. 
Note that $H$ is an ample $\mathbb Q$-divisor. 
We put $B=(1/l)A$. 
Then $x\not\in B$ and 
$(X, \Delta +D+\varepsilon B)$ 
is not 
log canonical 
at $x'$ for every $\varepsilon>0$. 
\end{proof} 

Let us start the proof of Proposition \ref{122}. 

\begin{proof}[{Proof of {\em{Proposition \ref{122}}}}] 
We use the notation in Remark \ref{re2}. 
Let $W_1$ (resp.~$W_2$) be the minimal lc center of $(X, \Delta)$ such that 
$x_1 \in W_1$ (resp.~$x_2\in W_2$) when 
$\mu_1=\mu (x_1, X, \Delta)<\dim X$ (resp.~$\mu_2=
\mu(x_2, X, \Delta)<\dim X$). 
Otherwise, we put $W_1=X$ (resp.~$W_2=X$). 
Possibly after switching $x_1$ and $x_2$, 
we can assume that $\mu_1\leq \mu_2$. Without loss of 
generality, we can assume that 
$c(k)\in \mathbb Q$ for every $k$. 
\begin{case}\label{ca1}
If $W_1\not \subset W_2$, then we can see that 
$x_1\not \in W_2$ (see Theorem \ref{a1} (2)). By Proposition \ref{145}, 
we can find an effective $\mathbb Q$-divisor $B\equiv N$ such that $x_1\not 
\in B$ and $(X, \Delta +\varepsilon B)$ is not log canonical 
at $x_2$ for every $\varepsilon >0$. In this case, we cut down 
the minimal lc center passing through $x_1$ as in 
the proof of Proposition \ref{12} by using 
Proposition \ref{14}. 
Then we obtain a $\mathbb Q$-divisor $D$ on $X$ satisfying 
the conditions (1), (2), and (3) in Proposition \ref{122}. 
\end{case}
\begin{case}\label{ca2}
If $W_1\subsetneq W_2$, then $x_2 \not \in W_1$. Thus, the above proof 
in Case \ref{ca1} 
works after switching $x_1$ and $x_2$. 
\end{case}
\begin{case} 
We assume $W_1=W_2$. 
If $W_1=W_2=X$, then we apply Proposition \ref{27}. 
Otherwise, we use Proposition \ref{144} and 
cut down $W_1=W_2$. 
Then, we obtain 
$(X, \Delta +G)$ such that 
\begin{itemize}
\item[(1)] $G$ is an effective $\mathbb Q$-divisor on $X$ and 
a rational number $d$ such that 
$G\equiv d\frac {\sqrt[\leftroot{0}\uproot{0} k]{2}k}{c(k)}N$ with $0<d<1$, 
where $k=\dim W_1$, 
\item[(2)] $(X, \Delta +G)$ is not kawamata log terminal at the 
points $x_1$ and $x_2$, and is log canonical at one of them, say at 
$x_1$, and 
\item[(3)] there is the minimal lc center $W'_1$ of $(X, \Delta +G)$ such that 
$x\in W'_1$ and $\dim W'_1<\dim W_1$. 
\end{itemize}
If $(X, \Delta +G)$ is log canonical 
at the both points $x_1$ and $x_2$, and $x_1$ and $x_2$ stay on the same new minimal lc center 
of $(X, \Delta +G)$, 
then we apply Proposition \ref{144} again. 
By repeating this process, we obtain the situation 
where there is a suitable effective 
$\mathbb Q$-divisor $G'$ on $X$ such that 
$(X, \Delta+G')$ is not log canonical at one of $x_1$ and 
$x_2$, or $x_1$ and $x_2$ are on different minimal lc centers 
of the pair $(X, \Delta +G')$. 
Then, we 
can apply the same arguments as in Case \ref{ca1} and Case \ref{ca2}. 
\end{case}
Thus, we finish the proof. 
\end{proof}

\section{Appendix}\label{sec-app} 

In this appendix, we collect some basic properties 
of lc centers and 
the new cohomologial technique (cf.~\cite{ambro}). 
Here, we do not explain 
the definition of quasi-log varieties because 
it is very difficult to grasp. 
We think that Theorem \ref{a2} is sufficient 
for our purpose in this paper. We recommend the reader interested 
in the theory of quasi-log varieties to see \cite{f-q}. 

Throughout this appendix, $X$ is a normal 
variety and $B$ is an effective $\mathbb Q$-divisor 
on $X$ such that 
$K_X+B$ is $\mathbb Q$-Cartier. 

\begin{thm}[{cf.~\cite[Proposition 4.8]{ambro}}]\label{a1} 
Assume that $(X, B)$ is log canonical. Then 
we have the following 
properties. 
\begin{enumerate}
\item $(X, B)$ has at most finitely 
many lc centers. 
\item An intersection of two lc centers 
is a union of lc centers. 
\item Any union of lc centers of $(X, B)$ is semi-normal. 
\item Let $x\in X$ be a closed point such that 
$(X, B)$ is log canonical but not kawamata log terminal at $x$. 
Then there is a unique minimal lc center $W_x$ 
passing through $x$, and $W_x$ is normal at $x$.  
\end{enumerate}
\end{thm}

The next theorem is one of the most important 
results in the theory of quasi-log varieties 
(see \cite[Theorem 4.4]{ambro}). 

\begin{thm}\label{a2} 
The pair $(X, B)$ has a natural quasi-log structure. 
Let $X\setminus X_{-\infty}$ be the maximal 
Zariski open set where 
$(X, B)$ is log canonical. 
Let $X'$ be the union of $X_{-\infty}$ and 
some lc centers of $(X, B)$. 
Then $X'$ has a natural 
quasi-log structure, which is induced by the quasi-log structure of 
$(X, B)$. 
We consider the following short exact sequence 
$$
0\to \mathcal I_{X'}\to \mathcal O_X\to \mathcal O_{X'}\to 0,  
$$ 
where $\mathcal I_{X'}$ is the defining ideal sheaf of 
$X'$ on $X$. 
Note that $X'$ is reduced on $X\setminus X_{-\infty}$. 
Assume that $X$ is projective. 
Let $L$ be a line bundle on $X$ such that 
$L-(K_X+B)$ is ample. 
Then $H^i(X, \mathcal I_{X'}\otimes L)=0$ for all  
$i>0$. 
In particular, the restriction map 
$$H^0(X, L)\to H^0(X', L|_{X'})$$ is 
surjective. 
\end{thm}

For the proofs of Theorem \ref{a1} and 
Theorem \ref{a2}, see \cite{ambro} and \cite[Section 3.2]{book}. 
We close this appendix with an important remark. 

\begin{rem}
The vanishing and torsion-free theorems required for the proofs of 
Theorem \ref{a1} and Theorem \ref{a2} can be 
proved easily by investigating mixed Hodge structures on compact cohomology groups of smooth varieties. Therefore, 
\cite{fuji-pro} is sufficient for our purposes. We do not 
need \cite[Chapter 2]{book}, which is very technical and 
seems to be inaccessible to non-experts. 
\end{rem}

\ifx\undefined\bysame
\newcommand{\bysame|{leavemode\hbox to3em{\hrulefill}\,}
\fi


\begin{thebibliography}{AS}

\bibitem[A]{ambro} 
F.~Ambro, 
Quasi-log varieties, 
Tr. Mat. Inst. Steklova {\textbf{240}} 
(2003), Biratsion. Geom. Linein. Sist. Konechno 
Porozhdennye Algebry, 220--239; translation 
in Proc. Steklov Inst. Math. 2003, no. 1 (240), 214--233 

\bibitem[AS]{as} 
U.~Angehrn, Y.-T.~Siu, 
Effective freeness and point separation for 
adjoint bundles, 
Invent. Math. {\textbf{122}} (1995), no. 2, 291--308. 

\bibitem[F1]{f-new} 
O.~Fujino, 
Effective base point free theorem for log canonical 
pairs---Koll\'ar type theorem, to appear in Tohoku Math. J. 

\bibitem[F2]{f-q} 
O.~Fujino, Introduction to the theory of quasi-log varieties, preprint (2007). 

\bibitem[F3]{book} 
O.~Fujino, Introduction to the log minimal 
model program for log canonical 
pairs, preprint (2009). 

\bibitem[F4]{fuji-pro} 
O.~Fujino, 
On injectivity, vanishing and torsion-free theorems 
for algebraic varieties, 
preprint (2008). 

\bibitem[F5]{fuji-non} 
O.~Fujino, 
Non-vanishing theorem for log canonical pairs, 
preprint (2009). 

\bibitem[Ka]{kawakita} 
M.~Kawakita, 
Inversion of adjunction on log canonicity, 
Invent. Math. {\textbf{167}} (2007), no. 1, 129--133.

\bibitem[Ko]{kollar2} 
J.~Koll\'ar, 
Singularities of pairs, Algebraic geometry---Santa 
Cruz 1995, 221--287, Proc. Sympos. Pure 
Math., {\textbf{62}}, Part 1, Amer. Math. Soc., 
Providence, RI, 1997. 
\end{thebibliography}
\end{document}